\documentclass[letterpaper, 10 pt, conference]{ieeeconf}  % Comment this line out
\IEEEoverridecommandlockouts                              % This command is only
\usepackage{graphics} % for pdf, bitmapped graphics files
\usepackage{amsmath} % assumes amsmath package installed
\usepackage{amssymb}  % assumes amsmath package installed
\usepackage{algorithm}

\newcommand{\cS}{\mathcal{S}}
\newcommand{\cK}{\mathcal{K}}
\newcommand{\cKH}{\cK_\mathcal{H}}
\newcommand{\cKR}{\cK_\mathcal{R}}
\newcommand{\cA}{\mathcal{A}}
\newcommand{\minuszero}{\setminus \{0\}}
\newcommand{\dual}[1]{#1^*}
\newcommand{\J}[1]{\partial f(#1)}
\newcommand{\innerproduct}[2]{{#1}^T #2}
\newcommand{\outerproduct}[2]{#1 {#2}^T}

\newcommand{\posvector}{\ensuremath{p}}
\newcommand{\posmatrix}{\ensuremath{P}}
\newcommand{\primalvector}{\ensuremath{r}}
\newcommand{\dualvector}{\ensuremath{h}}
\newcommand{\domright}{\ensuremath{\bar{r}}}
\newcommand{\domleft}{\ensuremath{\bar{h}}}
\newcommand{\arbdomright}{\ensuremath{\hat{r}}}
\newcommand{\arbdomleft}{\ensuremath{\hat{h}}}
\newcommand{\orienteddomright}{\ensuremath{\tilde{r}}}
\newcommand{\orienteddomleft}{\ensuremath{\tilde{h}}}
\newcommand{\orienteddomrights}{\ensuremath{\tilde{R}}}
\newcommand{\orienteddomlefts}{\ensuremath{\tilde{H}}}

\newcommand{\timestep}{\ensuremath{\tau}}
\newcommand{\widening}{\ensuremath{w}}
\newcommand{\funcsysandtime}[1]{\ensuremath{#1_{A_i, \timestep_i}}}
\newcommand{\discretized}{\ensuremath{\funcsysandtime{\tilde{A}}}}
\newcommand{\widenednaive}{\ensuremath{\funcsysandtime{W}}}
\newcommand{\funcsysandtimeandwidening}[1]{\ensuremath{#1_{A_i, \timestep_i, \widening_i}}}
\newcommand{\widened}{\ensuremath{\funcsysandtimeandwidening{W}}}
\newcommand{\cW}{\ensuremath{\mathcal{W}_{\cA,\timestep,\widening}}}

\DeclareMathOperator{\interior}{int}
\DeclareMathOperator{\conichull}{conic-hull}

\newtheorem{proposition}{Proposition}
\newtheorem{definition}{Definition}
\newtheorem{remark}{Remark}
\usepackage{tikz}
\usepackage{tkz-graph}  % for directed graphs, see https://tex.stackexchange.com/questions/37185/typesetting-a-directed-weighted-graph-with-tikz

\title{\LARGE \bf
Finding cones for K-cooperative systems
}

\author{Dimitris Kousoulidis and Fulvio Forni%
\thanks{D. Kousoulidis is supported by the Engineering and Physical Sciences Research Council (EPSRC) of the United Kingdom.}%
\thanks{D. Kousoulidis and F. Forni are with the Department of Engineering, University of Cambridge, CB2 1PZ, UK {\tt\small dk483@eng.cam.ac.uk \& f.forni@eng.cam.ac.uk}}
}

\begin{document}

\maketitle
\thispagestyle{empty}
\pagestyle{empty}

%%%%%%%%%%%%%%%%%%%%%%%%%%%%%%%%%%%%%%%%%%%%%%%%%%%%%%%%%%%%%%%%%%%%%%%%%%%%%%%%
\begin{abstract}
We design and test a cone finding algorithm to robustly address nonlinear system analysis through differential positivity.
The approach provides a numerical tool to study multi-stable systems, beyond Lyapunov analysis. 
The theory is illustrated on
two examples: a consensus problem with some repulsive interactions and second order agent dynamics, and a controlled duffing oscillator.
\end{abstract}

%%%%%%%%%%%%%%%%%%%%%%%%%%%%%%%%%%%%%%%%%%%%%%%%%%%%%%%%%%%%%%%%%%%%%%%%%%%%%%%%
\section{INTRODUCTION}

Despite the ubiquity of multi-stable systems (including, for example, bi-stable switches)
in diverse areas of engineering,
there is a lack of general tools for their analysis.
Analyzing their behavior often involves finding ways of using Lyapunov theory confined to the basin of attraction of each equilibrium.
This makes the process cumbersome and leads to local/regional stability results.
A similar situation is encountered when trying to apply traditional Lyapunov stability theory to mono-stable scenarios where the equilibrium state depends on the parameters.
Differential analysis \cite{forni_differential_2014} overcomes these limitations by using the linearized dynamics to characterize the behavior of a system
independently of the location of the system attractor.
In contraction theory, for example, convergence to a unique fixed point is established 
from the stability of the linearized dynamics 
computed along any system trajectory,
using tools such as matrix measures or Lyapunov inequalities applied to the linearized vector field
\cite{lohmiller_contraction_1998-1,russo_global_2010,simpson-porco_contraction_2014,forni_differential_2014-1}. 

We approach the problem from the perspective of differential positivity \cite{forni_differentially_2016}, 
which extends differential analysis approaches to multi-stable systems.
This theory enables us to reduce the analysis of the system to the problem of finding a cone that is forward invariant under some linear operators.
The existence of a suitable cone guarantees that almost every bounded 
trajectory converges to a fixed point.
In this paper we address this problem algorithmically, proposing a numerical method to build a cone that is invariant (and contracting) for the linearized dynamics of the system.

On linear systems, differential positivity corresponds to classical positivity
\cite{luenberger_introduction_1979,farina_positive_2000,rantzer_scalable_2015}.
In the nonlinear setting, differential positivity shows tight connections with 
monotone systems \cite{smith_monotone_1995, hirsch_chapter_2006,angeli_monotone_2003},
whose salient feature is that their trajectories preserve a partial order on the system 
state space.
The connection builds on the fact that every cone induces a partial order on the system state space and that positivity of the 
linearization entails  order preservation among systems trajectories. In this sense, monotonicity
is an integral version of differential positivity.

This connection with monotone systems is particularly relevant because, traditionally, most of the literature takes monotonicity as an intrinsic property
of a system. However, certifying monotonicity for a system is hard, with basic tests known
for the case of orthant positivity (Metzler Jacobian, conditions based on graph connectivity)
\cite{angeli_detection_2004,angeli_multi-stability_2004,blanchini_structural_2014}. 
In this sense, our algorithm provides a way to establish
monotonicity for systems that so far have remained elusive to the property.

At the most general level, differential positivity is defined with respect to a cone \emph{field} on manifolds and can also capture limit cycle behavior.
In this paper, however, we restrict ourselves to strict variants of differential positivity with constant cone fields.
On vector spaces,
finding a conic set that is forward invariant and contracting for the linearized dynamics has the practical interpretation that rays on the boundary of the cone are mapped into its interior. 
This form of projective contraction, connected to Perron-Frobenius theory \cite{bushell_hilberts_1973}, 
guarantees that the asymptotic behavior of the system is essentially one dimensional
and, for systems on a vector space, leads to trajectories that almost globally converge to some fixed point.

To obtain conditions we can test more easily and to make an explicit distinction with the general theory of differential positivity, while maintaining the strong convergence results, we introduce the notion of \emph{strict K-cooperativity}.
The name is picked to highlight the connection with K-cooperative systems \cite{hirsch_chapter_2006}, a large subclass of monotone systems.

Finding a cone that is forward invariant under some linear operators is a difficult task computationally \cite{protasov_when_2010}.
In this paper we derive basic necessary conditions for the existence of such a cone, then we 
leverage the classical theory of positive matrices \cite{berman_nonnegative_1989} to 
algorithmically address its construction.
The algorithm is illustrated and tested
on a consensus dynamics problem with nonlinear
and repulsive interactions, and on a bistable electro-mechanical system. 
We find cones for different parameter values and study their robustness to static uncertainties.
The examples show the potential of our approach in nonlinear analysis.

{
	\small
	\vspace{3mm}
	\textbf{Notation:}
	A matrix \posmatrix\ is non-negative, $\posmatrix \geq 0$  (positive, $\posmatrix > 0$), 
	if all its elements are non-negative (positive).
	The interior of a set $\cS$ is denoted by $\interior(\cS)$.
	A proper cone is a set $\cK$ such that: (i) if $\primalvector_1,\, \primalvector_2 \in \cK \text{ and }0 \leq \posvector_1,\,\posvector_2 \in \mathbb{R}$, then $\posvector_1 \primalvector_1 + \posvector_2 \primalvector_2 \in \cK$; (ii) $\interior(\cK) \neq \{0\}$; and (iii) if $\primalvector \in \cK$, then $-\primalvector \notin \cK$.
	The dual cone of a cone $\cK$ is denoted by $\dual{\cK} = \{\dualvector: \innerproduct{\dualvector}{\primalvector} \geq 0 ,\, \forall \primalvector \in \cK\}$.
	Considering two cones $\cK_1$ and $\cK_2$, $\cK_1 \subseteq \cK_2$ denotes the usual set inclusion. We will use $\cK_1 \subset \cK_2$ to denote $\primalvector \in \cK_1 \minuszero \implies \primalvector \in \interior(\cK_2)$.
}

%%%%%%%%%%%%%%%%%%%%%%%%%%%%%%%%%%%%%%%%%%%%%%%%%%%%%%%%%%%%%%%%%%%%%%%%%%%%%%%%
\section{POSITIVITY AND MONOTONICITY}
\label{sec:background}

\subsection{Positivity}
A linear dynamical system is \emph{positive} if trajectories starting in a cone remain in the cone. 
Namely, a continuous time system $\dot{x} = Ax$, $x\in \mathbb{R}^n$, 
is positive with respect to the proper cone $\cK$ if all its trajectories $x(\cdot)$ satisfy
\begin{equation}
\label{eq:positivity}
x(0) \in \cK \implies x(t) \in \cK \text{ for all } t \geq 0.
\end{equation}
The classical example is given by systems whose state vector is constrained to be element-wise positive. 
The cone corresponds to the positive orthant $\cK = \mathbb{R}^n_+$.

The trajectories of strictly positive systems ($x(0) \in \cK \minuszero \implies x(t) \in \interior(\cK) \text{ for all } t > 0$) 
converge to a ray. This qualitative behavior makes positive systems central in engineering
\cite{luenberger_introduction_1979,farina_positive_2000,rantzer_scalable_2015}. 
Convergence to a ray is a consequence of the Perron-Frobenius 
theorem which states that every strictly positive map $A$ admits a \emph{strictly dominant eigenvalue}, that is, in the continuous time case,
a right-most simple eigenvalue  \cite[Ch. 6]{luenberger_introduction_1979}. 
The  associated eigenvector is the \emph{dominant eigenvector}. The Perron-Frobenius theorem can be seen as a consequence of the Banach contraction 
mapping theorem: for any positive $t>0$, the semiflow $e^{At}$ maps any ray in $\cK$ into the interior of $\cK$, which guarantees
contraction of Hilbert's  metric \cite{bushell_hilberts_1973}. Thus, $e^{At}$ has a fixed point, the dominant eigenvector, lying within $\cK$ \cite{berman_nonnegative_1989, vandergraft_spectral_1968}. 

There are two main approaches to certify strict positivity for closed linear systems: one requires finding a right-most isolated eigenvalue in $A$;
the other searches for a cone that satisfies \eqref{eq:positivity}.
The algorithm we present in Section \ref{sec:cfa} approaches the latter numerically
and, in contrast to other methods, allows for extensions to the nonlinear setting.

\subsection{Monotonicity and K-cooperativity}
A dynamical system is \emph{monotone} if its trajectories preserve some partial ordering $\preceq$ on the system state space.
Namely, for any pair of trajectories $x_1(\cdot)$ and $x_2(\cdot)$, monotonicity requires:
\begin{equation}
x_1(0) \preceq x_2(0) \implies x_1(t) \preceq x_2(t),
\end{equation}
 for any $t \geq 0$. Under mild conditions, almost all bounded trajectories of a monotone system converge to 
 fixed points \cite{hirsch_chapter_2006}, a feature that makes them central in system theory and feedback control \cite{smith_monotone_1995, angeli_monotone_2003, mallet-paret_poincare-bendixson_1990, leenheer_tutorial_2004}.

Our interest in monotone systems stems from the fact that any
proper cone $\cK$ induces a partial order:
\begin{equation}
\label{eq:cipo}
x_1 \preceq_{\cK} x_2 \iff  x_2 - x_1 \in \cK
\end{equation}

Equation \eqref{eq:cipo} suggests a new characterization for monotonicity,
based on system linearization \cite{forni_differentially_2016}:
\begin{definition}[Strict K-cooperativity]
	\label{def:k-coop}
	Let $\cK\subseteq \mathbb{R}^n $ be a proper cone.
	A system $\dot{x} = f(x)$ is strictly K-cooperative with respect to $\mathcal{K}$ if for all $x \in \mathcal{X},\,\delta x \in \cK \minuszero, \,\dualvector \in \dual{\cK} \minuszero,$ the \emph{strict sub-tangentiality} condition is met:
	\begin{equation}
	\innerproduct{\dualvector}{\delta x} = 0 \implies \innerproduct{\dualvector}{\J{x}\delta x} > 0
	\label{eq:strict-subtang}
	\end{equation}
\end{definition}

This implies that, for a strictly K-cooperative system, any trajectory $(x(\cdot),\delta x(\cdot))$ of the prolonged system
\begin{equation}
\dot{x} = f(x) \qquad \dot{\delta x} = \J{x} \delta x,
\label{eq:diff+top}
\end{equation}
satisfies
\begin{equation}
\delta x(0) \in \cK \minuszero
\implies
\delta x(t) \in \mbox{int}(\cK) \mbox{ for } t > 0
\label{eq:diff+}
\end{equation}

\begin{proposition}
Any strictly K-cooperative system is strictly differentially positive.
\end{proposition}
\begin{proof}
From \eqref{eq:diff+top} and \eqref{eq:diff+}, \cite[Definition 2]{forni_differentially_2016} holds for all $T > 0$.
\end{proof}

Strictly K-cooperative systems enjoy strong convergence properties, as shown by the theory of differentially positive systems \cite[Corollary 5]{forni_differentially_2016}:
\begin{proposition}[Steady state of K-cooperative systems]
	Consider a strictly K-cooperative dynamical system $\dot{x} = f(x)$, $x \in \mathbb{R}^n$. 
	Suppose that all trajectories are bounded. Then, for almost all $x \in \mathbb{R}^n$, the $\omega$-limit set $\omega(x)$ is a fixed point.
	\label{prop:conv_main}
\end{proposition}

Proving the existence of a cone $\cK$ to certify K-cooperativity is a difficult task.
In this paper we propose an algorithm that, given a system, finds a cone $\cK$ that satisfies \eqref{eq:strict-subtang}.
Indeed, by Proposition \ref{prop:conv_main}, the algorithm guarantees that the system's trajectories almost globally converge to some fixed point.

%%%%%%%%%%%%%%%%%%%%%%%%%%%%%%%%%%%%%%%%%%%%%%%%%%%%%%%%%%%%%%%%%%%%%%%%%%%%%%%%
\section{POLYHEDRAL CONES}
\label{sec:testing}

\subsection{Representations, Membership, and Tests}
We work with proper polyhedral cones adopting 
the following representations:
\begin{LaTeXdescription}
	\item[H-representation] for any matrix $H$, 
	\begin{equation}
	\cKH(H) = \{\primalvector: H\primalvector \geq 0 \} \label{eq:h-rep} \
	\end{equation}
	\item[R-representation] for any matrix $R$,
	\begin{equation}
	\cKR(R) = \{\primalvector: \primalvector = R\posvector, \posvector\geq 0\} \label{eq:r-rep} \
	\end{equation}
\end{LaTeXdescription}
Strict inequalities characterize the interior of the cones.
In the H-representation, the rows of $H$ define half-spaces through the origin, with the overall cone given by their intersection. 
In the R-representation, the columns of $R$ are generator rays and the overall cone is their conical hull.
For a proper cone we need at least $n$ independent columns in $R$ and $n$ independent rows in $H$, where $n$ is the dimension of the space.

By the Minkowski-Weyl theorem \cite{fukuda_lecture_2016}, any polyhedral cone admits both representations.
However, in dimensions above three, it can be computationally expensive to convert between the two representations.
From this point onward, we focus on R-representation cones, noting that the analysis can be easily extended to H-representation cones using duality.

\begin{proposition}
	\label{prop:tests}
	Suppose $\cK = \cKR(R)$ for some matrix $R$.
	An $n$ dimensional system $\dot{x} = f(x) ,\, x \in \mathcal{X}$, is strictly K-cooperative with respect to $\cK$ if for all $x \in \mathcal{X}$, 
		there exists a matrix $\posmatrix > 0 $ such that, for some $\alpha$,
		\begin{equation}
		(\alpha I + \J{x})R = R\posmatrix
		\label{eq:diff_r_test}
		\end{equation}
\end{proposition}
\begin{proof}
	Let $\delta x \in \cK \minuszero,\, \dualvector \in \dual{\cK} \minuszero,\, \innerproduct{\dualvector}{\delta x} = 0$, and $\cKR(R)$ is proper.
	By definition, $\delta x = Rp$ for some element-wise positive vector $0 \neq \posvector \geq 0$.
	Note that at least one element of $\posvector$ is strictly greater than zero.

	Then, for any $\alpha \in \mathbb{R}$, $\innerproduct{\dualvector}{\J{x} \delta x} = \innerproduct{\dualvector}{(\alpha I + \J{x}) \delta x} = \innerproduct{\dualvector}{(\alpha I + \J{x}) R \posvector} = \innerproduct{\dualvector}{R \posmatrix\posvector}> 0$ and \eqref{eq:strict-subtang} holds.
	The strict inequality follows from the fact that
	$0 \neq \posvector \geq 0$, $0 \neq \dualvector^T R \geq 0$, and $\posmatrix > 0$. 
\end{proof}

Proposition \ref{prop:tests} provides simple geometric conditions for K-cooperativity
with respect to any polyhedral cone. These conditions need to hold for  every $x \in \mathcal{X}$,
which can be tested numerically through conical relaxations, as shown in Section \ref{sec:conical_approx}. 

\subsection{Conical Relaxation} \label{sec:conical_approx}

The tests in Proposition \ref{prop:tests} can be made numerically tractable by finding a finite family of matrices $\cA := \{A_1, A_2, \dots, A_k\}$ such that, for all $x \in \mathcal{X}$:

\begin{equation}
\J{x} \in \conichull(\cA) \minuszero, \label{eq:coni-cond}
\end{equation}
where, for finite $k$:
$$
\conichull(\cA) := \left\{ A : A = \sum_{i=1}^k \posvector_i A_i ,\, \posvector_i \geq 0\right\}
$$

A set $\cA$ that satisfies \eqref{eq:coni-cond} is referred to as a \emph{conical relaxation} of the dynamics $\dot{x} = f(x)$.

\begin{proposition}
	Suppose $\cK = \cKR(R)$ for some matrix $R$.
	An $n$ dimensional system $\dot{x} = f(x) ,\, x \in \mathcal{X}$, is strictly K-cooperative with respect to $\cK$ if \eqref{eq:coni-cond} holds and, for all $A_i \in \cA$, 
	there exists a matrix $\posmatrix_i > 0$ such that
		\begin{equation}
		(\alpha_i I + A_i)R = R\posmatrix_i
		\label{eq:coni_r_test}
		\end{equation}
	for some $\alpha_i \in \mathbb{R}$.
	\label{prop:coni_tests}
\end{proposition}
\begin{proof}
	Since \eqref{eq:coni-cond} holds, $\J{x}$ can be written as:
	\begin{equation*}
	\J{x} = \sum_{i=1}^k \posvector_i^x A_i,
	\end{equation*}
	for some $\posvector_i^x \geq 0$ and $\sum_{i=1}^k \posvector_i^x \neq 0$. 

	Define $\alpha := \sum_{i=1}^k \alpha_i \posvector_i^x$ and $\posmatrix := \sum_{i=1}^k \posvector_i^x \posmatrix_i$, and note that $\posmatrix > 0$. Then:
	\begin{align*}
	& (\alpha I + \J{x}) R  = \left(\sum_{i=1}^k \alpha_i \posvector_i^x I + \sum_{i=1}^k \posvector_i^x A_i \right)R = \\
	&= \sum_{i=1}^k \posvector_i^x (\alpha_i I + A_i) R = \sum_{i=1}^k \posvector_i^x R\posmatrix_i 
	 = R\sum_{i=1}^k \posvector_i^x\posmatrix_i = R\posmatrix
	\end{align*}
	Thus, strict K-cooperativity follows from Proposition \ref{prop:tests}.
\end{proof}

Proposition \ref{prop:coni_tests} shows that
K-cooperativity with respect to a given polyhedral cone can be determined by solving $k$ Linear Programming (LP) problems. 
However, some methods of building the family of matrices $\cA$ lead to a combinatorial explosion in the number of matrices to test.
General methods for producing tight conical relaxations are left as future work.

Proposition \ref{prop:coni_tests} also provide ways to test K-cooperativity for systems subject to uncertainties. Suppose that \eqref{eq:coni_r_test} holds for $\dot{x} = f(x)$.
Then it also holds for any perturbed system $\dot{x} = f(x) + g(x)$ where $\partial g(x) \in \conichull(\cA)$.
Moreover, in a situation where the uncertainties in the system can be represented in the linearizations by $\J{x} + \mathcal{Q}$, where $\mathcal{Q}$ is a given family of perturbations, the family of matrices $\cA_\mathcal{Q}$ can be suitably adapted to the perturbations. If $\J{x} + Q \in \conichull(\cA_\mathcal{Q})$, then \eqref{eq:coni_r_test} with $\cA_\mathcal{Q}$ guarantees strict K-cooperativity of the perturbed system.

Given a conical relaxation, testing K-cooperativity with respect to a given cone is a tractable problem. However, deciding whether or not a system is K-cooperative with respect to \emph{some} cone is a much harder question, even when a conical relaxation is used \cite{protasov_when_2010}.
The work in Section \ref{sec:finding_cones} is a first attempt in this direction.

%%%%%%%%%%%%%%%%%%%%%%%%%%%%%%%%%%%%%%%%%%%%%%%%%%%%%%%%%%%%%%%%%%%%%%%%%%%%%%%%
\section{FINDING CONES} \label{sec:finding_cones}

\subsection{Necessary Conditions} \label{sec:nec}
We present two necessary conditions for the existence of a cone that satisfies \eqref{eq:coni_r_test}. The first is a spectral condition that can be verified for every $A_i \in \cA$ individually, summarized in the following Proposition:

\begin{proposition}[Spectral Condition]
	\eqref{eq:coni_r_test} can hold for a given $\cA$ only if every $A_i \in \cA$ has a strictly dominant eigenvalue.
	\label{prop:spectral}
\end{proposition}
\begin{proof}
	As in the proof of Proposition \ref{prop:tests}, \eqref{eq:coni_r_test} implies that every $A_i \in \cA$ satisfies the strict sub-tangentiality condition with respect to $\cK$.

	This implies that a system $\dot{z} = A_i z$ is strictly positive with respect to $\cK$, which in turn requires $A_i$ to have a strictly dominant eigenvalue (see \cite[Theorems 4.3.37 and 4.3.41]{berman_nonnegative_1989} for detailed proofs).
\end{proof}

The second is a geometric condition on the compatibility between the invariant spaces of 
each $A_i \in \cA$ based on the following property:

\begin{proposition}
	For any cone $\cK$ that satisfies \eqref{eq:coni_r_test}, every $A_i \in \cA$ must have its dominant right eigenvector $\domright_i$ in $\interior(\cK)$ and its dominant left eigenvector $\domleft_i$ in $\interior(\dual{\cK})$.
	\label{prop:geo_ind}
\end{proposition}
\begin{proof}
	Continuing from the proof of Proposition \ref{prop:spectral}, if a system $\dot{z} = A_i z$ is strictly positive with respect to $\cK$, a dominant right eigenvector of $A_i$ must be in $\interior(\cK)$ \cite[Theorem 4.3.34]{berman_nonnegative_1989}.
	Additionally, the dual system $\dot{\eta} = A_i^T \eta$ must be strictly positive with respect to $\dual{\cK}$ \cite[Theorem 4.3.45]{berman_nonnegative_1989}.
	As such, a dominant right eigenvector of $A_i^T$ must be in $\interior(\dual{\cK})$. This is a dominant left eigenvector of $A_i$.
\end{proof}

We observe that the property outlined in Proposition \ref{prop:geo_ind} refers to a common cone $\cK$, therefore,
it can be used to characterize a necessary condition based on the construction of two useful cones $\cK_{inner}(\cA)$ and $\cK_{outer}(\cA)$.
These cones will be used to initialize and terminate our cone finding algorithm (Section \ref{sec:cfa}).
The construction of $\cK_{inner}(\cA)$ and $\cK_{outer}(\cA)$ is detailed below,
building on the `Orientation Trick' of \cite{forni_path-complete_2017}.

\begin{algorithm}[H]
	\caption{Inner and Outer Cones} \label{alg:inner_outer_cones}
	\begin{itemize}
		\item For each $A_i \in \cA$, define $\arbdomright_i$ and $\arbdomleft_i$ as arbitrary orientations of the right and left dominant eigenvectors of $A_i$
		\item We will build two matrices $\orienteddomrights \in \mathbb{R}^{n \times l}$ and $\orienteddomlefts \in \mathbb{R}^{l \times n}$
		\item The first row of $\orienteddomlefts$, $\orienteddomleft_1^T$, is set to $\arbdomleft_1^T$
		\item The $k^{th}$ column of $\orienteddomrights$, $\orienteddomright_k$, is set to either $\arbdomright_k$ or $-\arbdomright_k$ such that $\innerproduct{\orienteddomleft_1}{\orienteddomright_k} \geq 0$
		\item For $j > 1$, the $j^{th}$ row of $\orienteddomlefts$, $\orienteddomleft_j^T$, is set to either $\arbdomleft_j^T$ or $-\arbdomleft_j^T$ such that $\innerproduct{\orienteddomleft_j}{\orienteddomright_j} \geq 0$
		\item $\orienteddomrights$ and $\orienteddomlefts$ then define
		$\cK_{inner}(\cA) := \cKR(\orienteddomrights)$
		and
		$\cK_{outer}(\cA) := \cKH(\orienteddomlefts)$
	\end{itemize}
\end{algorithm}

\begin{proposition}[Geometric Condition]
	\eqref{eq:coni_r_test} can hold for a given $\cA$ only if $\cK_{inner}(\cA) \subset \cK_{outer}(\cA)$
	\label{prop:geo}
\end{proposition}
\begin{proof}
	Assume \eqref{eq:coni_r_test}. This implies that Proposition \ref{prop:geo_ind} holds for a given family of matrices $\cA$ and cone $\cK$. 
	Consider the left and right dominant eigenvectors of Proposition \ref{prop:geo_ind}.
	Then $\innerproduct{\domleft_j}{\domright_k} > 0$ for all $(j,k)$.
	
	If $\orienteddomleft_1 = \domleft_1$, by construction $\orienteddomright_k = \domright_k$ and $\orienteddomleft_j = \domleft_j$.
	If $\orienteddomleft_1 = -\domleft_1$, by construction $\orienteddomright_k = -\domright_k$ and $\orienteddomleft_j = -\domleft_j$.
	In both cases $\innerproduct{\domleft_j}{\domright_k} > 0 \iff \innerproduct{\orienteddomleft_j}{\orienteddomright_k} > 0$.
	But if $\cK_{inner}(\cA) \not\subset \cK_{outer}(\cA) $ there exists some $(j,k)$ such that $\innerproduct{\orienteddomleft_j}{\orienteddomright_k} \leq 0$; which leads to a contradiction.
\end{proof}

Note that checking whether $\cK_{inner}(\cA) \subset \cK_{outer}(\cA)$ amounts to computing $\orienteddomlefts$ and $\orienteddomrights$ and verifying that $\orienteddomlefts\orienteddomrights > 0$. As such, it is a very scalable test.

\begin{remark}
Because of the complexity of producing tight conical relaxations and finding cones, these two conditions can be used to quickly rule out K-cooperativity by testing sets of randomly sampled $\J{x}$ of a candidate system.
\end{remark}
                 
\subsection{Cone Finding Algorithm}
\label{sec:cfa} 
A simple procedure to find a cone that satisfies \eqref{eq:coni_r_test} for a given $\cA$ is described below.
We begin by illustrating the main ideas on the easier case of a single matrix $A$.
Given an initial cone $\cKR(R^{(0)})$ and a linear operator $W$, consider the matrix $R^{(1)}$ formed by combining the columns of $R^{(0)}$ and $WR^{(0)}$, $R^{(1)} = [R^{(0)}, WR^{(0)}]$. This can be repeated recursively such that $R^{(k+1)} = [R^{(k)}, WR^{(k)}]$.
If $W := (\alpha I + A)$, $R^{(k+1)}$ will trivially satisfy $(\alpha I + A)R^{(k)} = R^{(k+1)} \posmatrix$ with $\posmatrix \geq 0$.
Moreover, the projective contraction property of positive systems tells us that if $\{A\}$ satisfies Proposition \ref{prop:spectral} and $\cK_{inner}(\{A\}) \subset \cKR(R^{(0)}) \subset \cK_{outer}(\{A\})$,
there will be an $\alpha$ for which the sequence of cones $R^{(k)}$ converges to a cone
$\cKR(R^{(k)}) \to \cK(R)$ in a finite number of steps. It follows that $\cKR(R)$ will satisfy \eqref{eq:coni_r_test} for $\{A\}$ and $\posmatrix \geq 0$.

This motivates an algorithm for finding cones for a family of systems. 
We will first define a \emph{time-step based parametrization}. 
Then, the technical requirement of satisfying \eqref{eq:coni_r_test} with $\posmatrix_i > 0$ is tackled by adding a \emph{widening operation} to $W_i$.
A procedure for \emph{initializing $R^{(0)}$} is also needed.
Finally, unlike the single matrix case, no guarantees about convergence can be made and additional \emph{termination conditions} need to be provided.

\emph{Time-step based parametrization:} 
We denote the dominant eigenvalue of a matrix $\{A_i\}$ satisfying Proposition \ref{prop:spectral} as $\lambda_i$.
The values of $\alpha_i$ for which the single matrix algorithm described above is guaranteed to converge are the values for which $(\alpha_i I + A_i)$ has a real simple positive eigenvalue with a larger absolute value than any other eigenvalue.
Such an eigenvalue will be referred to as an \emph{absolutely dominant eigenvalue} of $(\alpha_i I + A_i)$.
We adopt the parametrization $\discretized := I + \tau_i(A_i - \lambda_i I)$ which corresponds to taking $\alpha_i = 1/\tau_i -\lambda_i$.
The only free parameter in this formulation is $\tau_i$, which acts like a discretization time-step.
For any $A_i$ satisfying Proposition \ref{prop:spectral},
there is a non-empty open interval $\tau_i \in (0, T_i)$ for which \discretized has a dominant eigenvalue.

\emph{Widening operation:}
Setting $\widenednaive := \discretized$ and using $\widenednaive$ to product the sequence of matrices $R^{(k)}$, we have that
$(\alpha_i I + A_i)R^{(k)} = R^{(k+1)} \posmatrix_i$ with $\posmatrix_i \geq 0$. 
However, for the strong convergence properties of strict K-cooperativity we require $\posmatrix_i > 0$.
With this aim, define \widened as:
\begin{equation}
	\widened := \discretized - \widening_i \outerproduct{\orienteddomright_i}{\orienteddomleft_i}, \label{eq:widening}
\end{equation}
where $\tau_i, \widening_i > 0$ and fixed, and $\orienteddomright_i$ and $\orienteddomleft_i$ are the right and left dominant eigenvectors of $A_i$ in $\cK_{inner}(\cA)$ and $\cK_{outer}(\cA)$ respectively.
Now, if a cone $\cKR(R^{(k+1)})$ is formed as $R^{(k+1)} = [R^{(k)}, \widened R^{(k)}]$,
then $\widened R^{(k)} = R^{(k+1)} \bar{\posmatrix}_i$ for some $\bar{\posmatrix}_i \geq 0$.
But, $\cK_{inner}(\cA) \subset \cKR(R^{(k)}) \subset \cK_{outer}(\cA)$, so $\innerproduct{\orienteddomleft_i}{R^{(k)}} > 0$, and 
$\orienteddomright_i = R^{(k+1)}\tilde{\posmatrix}_i$ for some $\tilde{\posmatrix}_i > 0$.
Hence, we get $\discretized R^{(k)} = \widening_i \outerproduct{\orienteddomright_i}{\orienteddomleft_i} + R^{(k+1)} \bar{\posmatrix}_i = R^{(k+1)} \posmatrix_i$ for some $\posmatrix_i > 0$. 
The additional \emph{widening} provided by $- \widening_i \outerproduct{\orienteddomright_i}{\orienteddomleft_i}$  can be interpreted as moving new rays slightly away from $\orienteddomright_i$,
and $\widening_i$ can be thought of as a `widening coefficient'. Given $\tau_i \in (0, T_i)$, there will be a maximum $\widening_i$ for which $\widened$ still has an absolutely dominant eigenvalue.
In what follows, the set of $\widened$ operators is denoted by $\cW$.

Now if a cone $\cKR(R^{(k+1)})$ is formed with $R^{(k+1)} = [R^{(k)}, W_1R^{(k)}, \dots, W_lR^{(k)}]$ 
with $W_i \in \cW$, there must then exist $\alpha_i$ such that $(\alpha_i I + A_i)R^{(k)} = R^{(k+1)}\posmatrix_i$ with $\posmatrix_i > 0$ for all $A_i \in \cA$.
As such, if $R^{(k)}$ converges to some matrix $R$ as $k \to \infty$, then $\cKR(R)$ satisfies \eqref{eq:coni_r_test}.

\emph{Initializing $R^{(0)}$:}
An initial cone $\cKR(R^{(0)})$ that satisfies $\cK_{inner}(\cA) \subset \cKR(R^{(0)}) \subset \cK_{outer}(\cA)$ can be generated by starting with $\cK_{inner}(\cA)$ and adding random vectors of small magnitude until $\cK_{inner}(\cA) \subset \cKR(R^{(0)})$, restarting with a smaller noise magnitude if at any point $\cKR(R^{(0)}) \not\subset \cK_{outer}(\cA)$.
We call this operation Initialize$(\cK_{inner}, \cK_{outer})$.

\emph{Termination conditions:}
The algorithm terminates if \eqref{eq:coni_r_test} is satisfied, which implies strict K-cooperativity. 
This is tested at every iteration.
If after an iteration of the algorithm the cone is no longer in the interior of $\cK_{outer}$, or if the maximum allowed number of iterations has been exceeded, the algorithm terminates without finding a cone.
Algorithm \ref{alg:sat} below summarizes the overall process:

\begin{remark}
	In the implementation of the algorithm we also remove redundant rays in $\cKR(R^{(k)})$ and test \eqref{eq:coni_r_test} every $N$ iterations for improved speed.
\end{remark}

\begin{algorithm}[H]
	{\bf Data:} The set of matrices $\cA$, \\
	the vector of time-steps $\timestep$, \\
	the vector of widening coefficients $\widening$, \\
	the maximum number of iterations $max\_iter$ \smallskip \\
	{\bf Result:} Cone satisfying \eqref{eq:coni_r_test} or $False$ \smallskip \\
	{\bf Procedure:} \\
	Compute $\cK_{inner}$ and $\cK_{outer}$ \, (Algorithm \ref{alg:inner_outer_cones}) \\
	Compute \cW, the collection of \widened in \eqref{eq:widening} \\
	$R^{(0)} := $ Initialize$(\cK_{inner}, \cK_{outer})$ \\
	\texttt{$k$} := 0 \\
	{\bf while }{$k < max\_iter$ and $ \cK_{R}(R^{(k)}) \subset \cK_{outer}$}: \\
	\hspace*{5mm} $R^{(k+1)} = [R^{(k)}, W_1R^{(k)}, \dots, W_lR^{(k)}]$ \, ($W_i \in \cW$) \\
	\hspace*{5mm} {\bf if }{K-cooperative($\cK_{R}(R^{(k+1)})$, $\cA$)}: \, (based on \eqref{eq:coni_r_test}) \\
	\hspace*{10mm} return $\cK_{R}(R^{(k+1)})$ \\
	\hspace*{5mm} $k = k + 1$ \\
	return $False$
	\caption{Cone Finding Algorithm \label{alg:sat}}
\end{algorithm}

%%%%%%%%%%%%%%%%%%%%%%%%%%%%%%%%%%%%%%%%%%%%%%%%%%%%%%%%%%%%%%%%%%%%%%%%%%%%%%%%
\section{EXAMPLES} \label{sec:examples}

\subsection{Consensus}
We consider K-cooperativity of standard consensus 
dynamics \cite{olfati-saber_consensus_2007,ren_distributed_2008} represented by:
\begin{equation}
\dot{x_i} = \sum_{j=1}^N f_{ij}(x_j-x_i) \qquad 0 < i \leq N,
\label{eq:cons_nonlinear} 
\end{equation}
where each agent $x_i$ is modeled by a simple integrator 
driven by the weighted differences with its neighboring agents,
characterized by functions
$f_{ij}: \mathbb{R} \to \mathbb{R}$ such that $f_{ij}(0) = 0$.
The system has a continuum of equilibria $x_e = \alpha 1_N$ 
where $\alpha \in \mathbb{R}$ and 
$1_N$ is the vector of $N$ ones.
The agents reach consensus when $x_1 = \dots = x_N$.

For linear positive weights convergence to consensus is
guaranteed by Perron-Frobenius theory 
\cite{luenberger_introduction_1979,sepulchre_consensus_2010}. 
This result extends to nonlinear strictly increasing weights: their slope
is strictly positive, the system is cooperative, and 
bounded trajectories converge to the consensus equilibria \cite{hirsch_chapter_2006}.
Convergence is also guaranteed for the case of unconstrained linear weights,
whenever $1_N$ is a dominant eigenvector of the system.

The presence of uncertainties or nonlinear weights that are not strictly
increasing makes the problem more challenging. However, from Proposition \ref{prop:conv_main},
we can use K-cooperativity to study consensus problems:
the trajectories of strictly K-cooperative consensus dynamics 
will converge to consensus whenever 
$1_N$ belongs to the interior of $\cK$ \cite[Section V]{forni_differential_2015}.
This demonstrates how Algorithm \ref{alg:sat} provides a numerical tool to study nonlinear and robust
consensus problems away from positive weights.

For illustration, we apply Algorithm \ref{alg:sat} to the network of agents 
in Fig. \ref{fig:cons_graph}, where the black edges represent linear 
weights normalized to $1$, and blue and red edges represents the nonlinear weights
$f_{15}$ and $f_{42}$, respectively.  
Using $A_{a,b}$ to denote the Jacobian of the consensus dynamics in Fig. \ref{fig:cons_graph},
for $f_{15}'=a$ and $f_{42}'=b$, we used Algorithm \ref{alg:sat} to find a common cone $\cK$
for the polytope of linearizations: $$\cA = \{ A_{(1,1)}, A_{(-1,1)}, A_{(1,-1)}, A_{(-0.9,-0.9)}\}$$
This polytope is visualized in Fig. \ref{fig:cons_polytope} and includes, for example:
$-1 \leq f_{15}' \leq 1$ and $f_{42}' = 1$, or $-0.9 \leq f_{15}' \leq 1$ and $-0.9 \leq f_{42}' \leq 1$. 

Consensus is thus reached when uncertainties of magnitude less than one 
affect the weight between nodes $1$ and $5$ or between nodes $4$ and $2$.
Consensus is also reached for any nonlinear weight between these
nodes whose slope is bounded between $-1$ and $1$. Indeed, $f_{15}(x_5-x_1) = \sin(x_5-x_1$)
and $f_{42}(x_2-x_4) = x_2-x_4$ is compatible with consensus. Finally, consensus is reached
when both $f_{15}$ and $f_{42}$ are nonlinear/uncertain provided that they are constrained
to a smaller negative range.

 \begin{figure}[tbp]
	\centering
	\begin{minipage}{.49\linewidth}
		\centering
		\begin{tikzpicture}[x=0.7cm,y=0.9cm]
		\SetUpEdge[lw         = 1.5pt,
		color      = black]
		\GraphInit[vstyle=Normal] 
		\SetGraphUnit{1.4}
		\tikzset{every node/.style={fill=yellow}}
		\tikzset{VertexStyle/.append  style={fill}}
		\Vertex{1}
		\NOEA(1){2}
		\SOEA(1){3}
		\EA(2){4}
		\EA(3){5}
		\tikzset{EdgeStyle/.style={->}}
		\Edge(2)(1)
		\Edge(4)(2)
		\Edge(3)(2)
		\Edge(1)(3)
		\Edge(5)(4)
		\Edge(1)(5)
		\tikzset{EdgeStyle/.style={->, color=blue, bend right}}
		\Edge[label=$f_{15}$, labelcolor=none, labelstyle={above, pos=.4}](5)(1)
		\tikzset{EdgeStyle/.style={->, color=red, bend right}}
		\Edge[label=$f_{42}$, labelcolor=none, labelstyle={below}](2)(4)
		\end{tikzpicture}
	\end{minipage}
\begin{minipage}{.49\linewidth}
	\includegraphics[width=\linewidth]{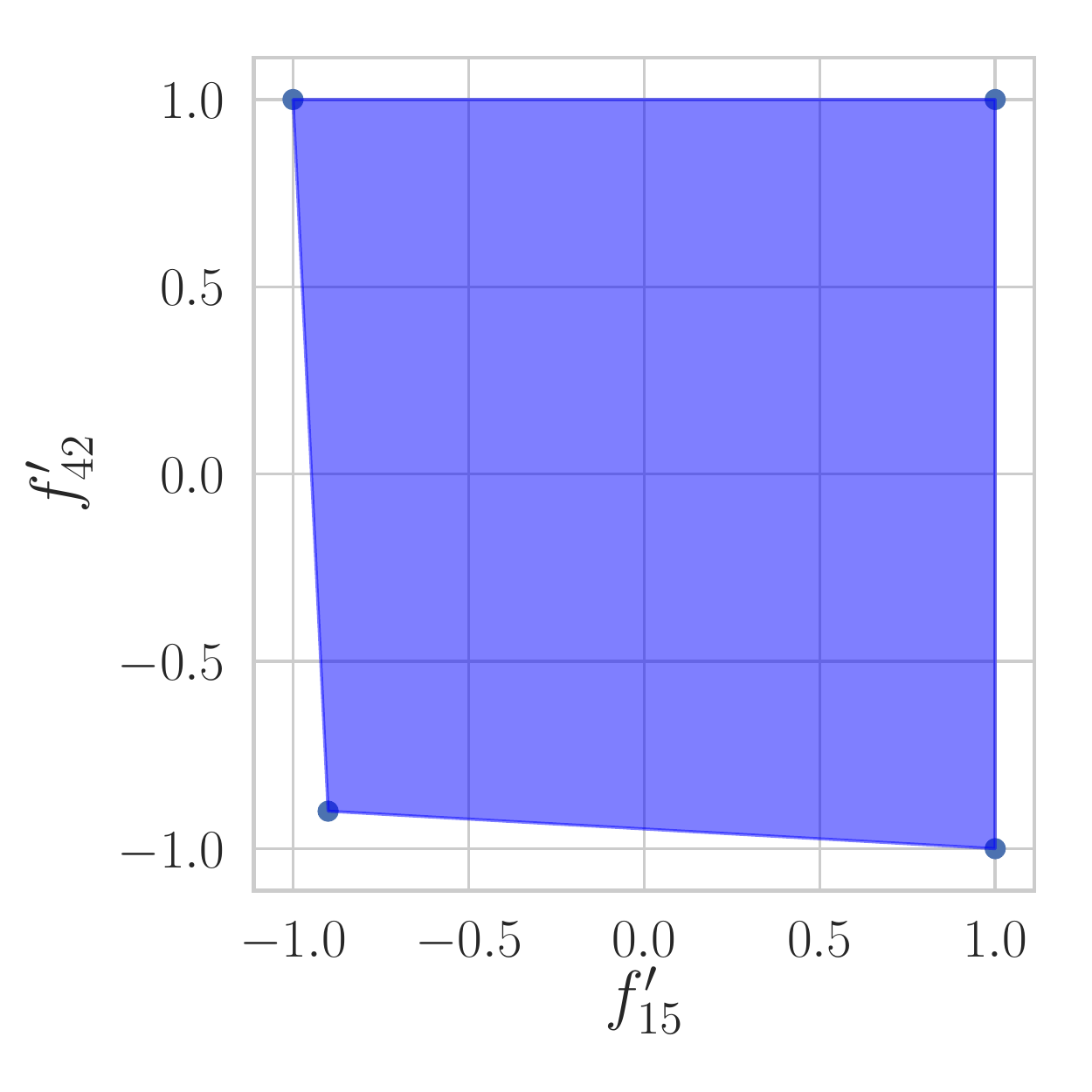}
	\end{minipage}
	\caption{\textbf{Left:} Consensus topology.
	\textbf{Right:} Polytope bounding $f_{15}'$ and $f_{42}'$.
	}
	\label{fig:cons_graph}
	\label{fig:cons_polytope}
\end{figure}

We complete the example by considering agent dynamics extended to second order networks of the form:
\begin{equation}
\label{eq:cons_nonlinear_two-state_standard} 
\dot{x}_i = v_i \,,\ 
\tau \dot{v}_i = - v_i + \sum_{j=1}^{N} f_{ij}(x_j-x_i)  \qquad 0 < i \leq N,
\end{equation}
where $\tau$ is a homogeneous time constant and $f_{ij}$ captures
the coupling between agents. We adopt the topology in Fig. \ref{fig:cons_graph}. 
For the second order case, consensus equilibria are given
by the subspace $(x,v)_e = (1_N, 0)$.

The necessary conditions in Propositions \ref{prop:spectral} and \ref{prop:geo} are not satisfied for $\tau \geq 1$.
A cone $\cK$ with $(1_N, 0) \in \cK$ was successfully found for $\tau = 0.3$.
Thus, the trajectories of \eqref{eq:cons_nonlinear_two-state_standard} with $\tau = 0.3$
asymptotically converge to consensus for any perturbed/nonlinear
$f_{15}$ and $f_{42}$ constrained to the intervals outlined above. An example trajectory with $\tau = 0.3$, $f_{15} = 0.9\sin(x_5-x_1)$, and $f_{42} = -0.9\sin(x_2-x_4)$ is shown in Fig. \ref{fig:cons_traj}.

\begin{figure}[tbp]
	\centering
	\includegraphics[width=\linewidth]{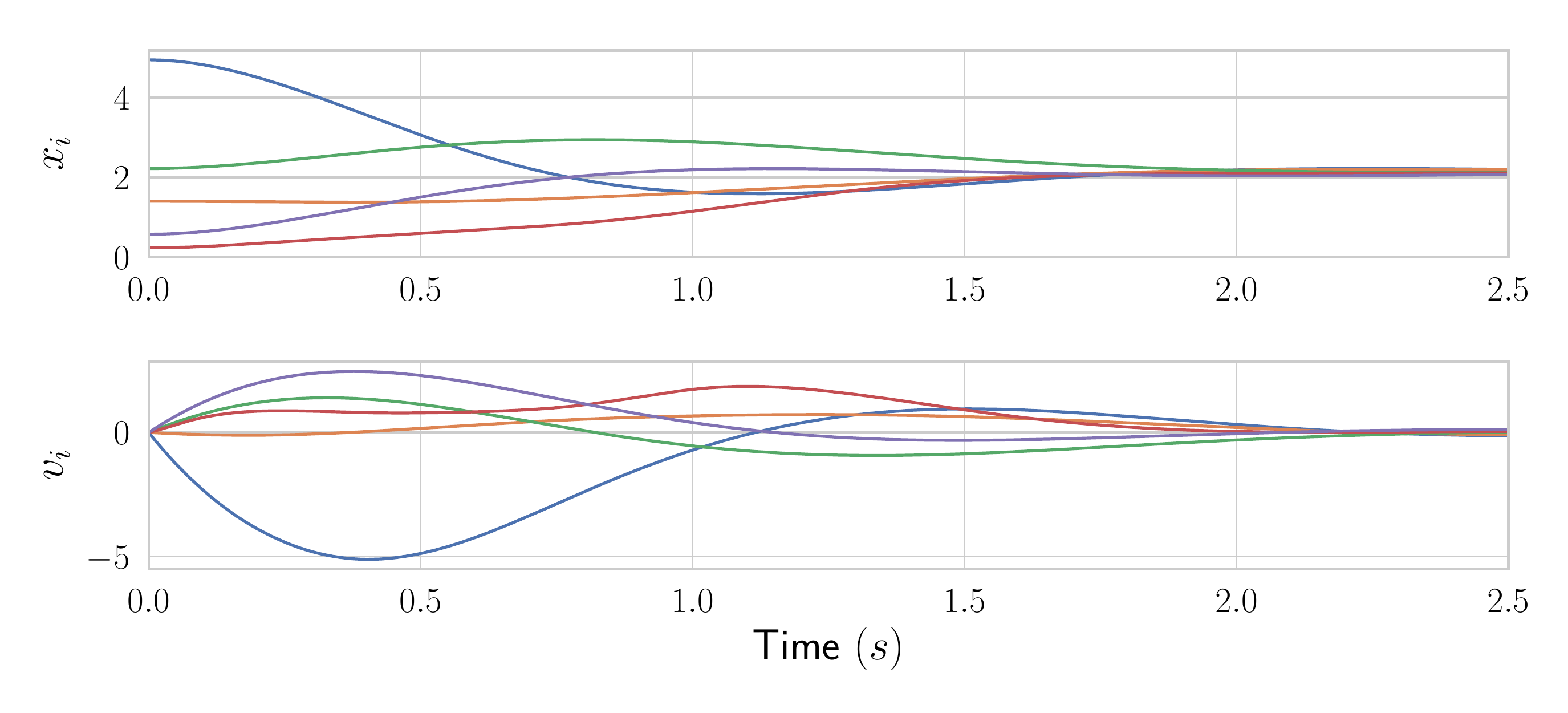}
	\caption{Sample trajectory for the second order consensus dynamics \eqref{eq:cons_nonlinear_two-state_standard} with $\tau = 0.3$, $f_{15} = 0.9\sin(x_5-x_1)$, and $f_{42} = -0.9\sin(x_2-x_4)$. Each color corresponds to an agent.}
	\label{fig:cons_traj}
\end{figure}

\subsection{Controlled Duffing Oscillator}
We study the three dimensional system given by:
\begin{equation}
\label{eq:spring}
\begin{split}
\dot{x_p} &= x_v \\
\dot{x_v} &= -\alpha(x_p) - c x_v + k_f x_i \\
L\dot{x_i} &= k_p (x_{ref} - x_p) -k_e x_v -R x_i
\end{split}
\end{equation}

This is based on the Duffing oscillator example in \cite{forni_differential_2018}.
The first two states capture the mechanics of the planar mechanical system and the DC motor inertia, while the third state is the electrical equation of the DC motor.
We use $\alpha(\cdot)$ to characterize a nonlinear spring.
Although simple, this system  captures a rich range of qualitative behaviors including mono-stability, bi-stability, and limit cycles.
We consider parameters $c = 5, k_f = 1, L = 0.1, k_e = 1, R = 1$.

Algorithm \ref{alg:sat} shows that \eqref{eq:spring} is strictly K-cooperative with respect to the cone shown in Fig. \ref{fig:spring_cone}
for $-2 \leq \alpha' \leq 5$ and $0 \leq k_p' \leq 3$.
This includes spring characteristics that can lead to bi-stability, as shown in Fig. \ref{fig:spring_traj} (right).

\begin{figure}[tbp]
	\centering
	\includegraphics[width=.49\linewidth]{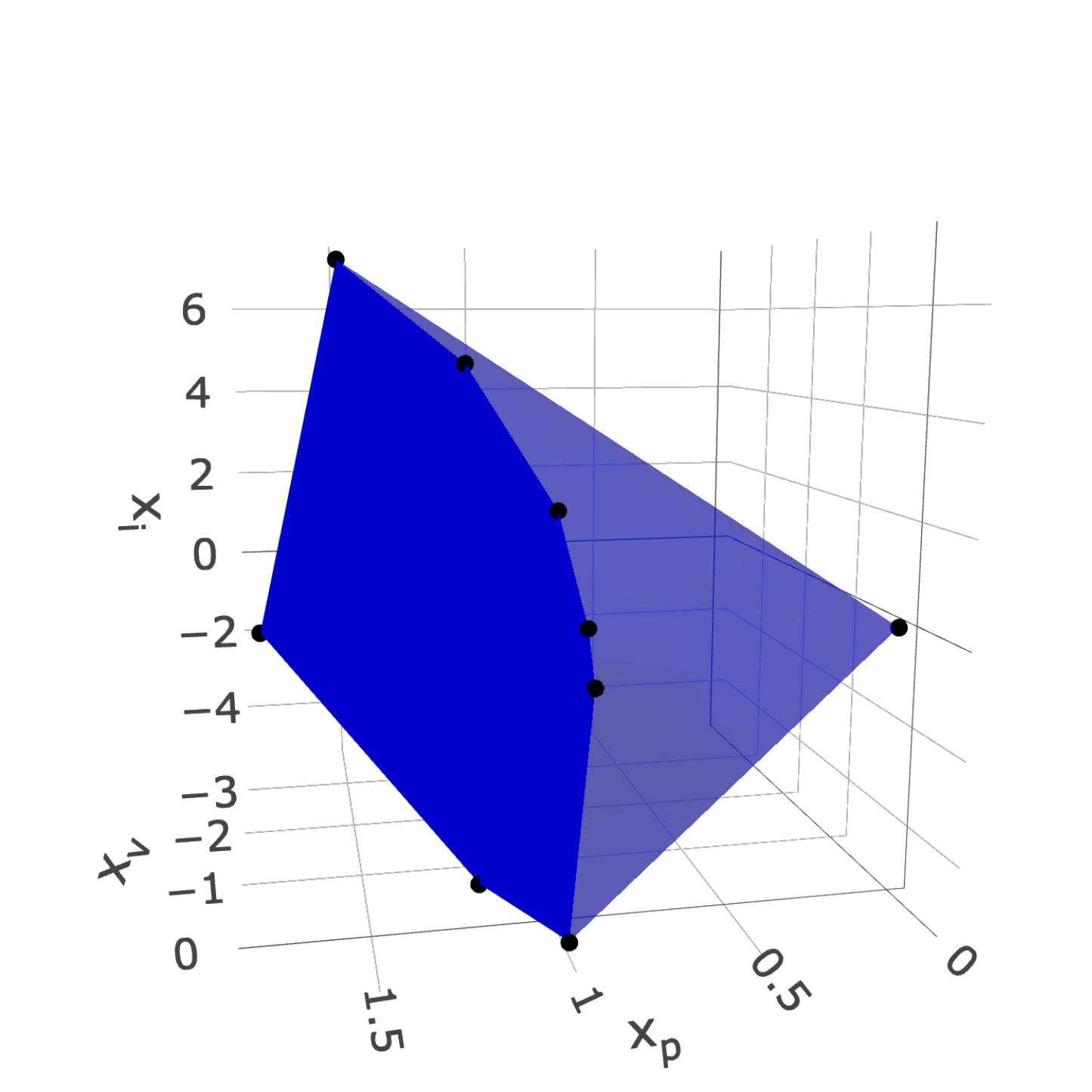}
	\includegraphics[width=.49\linewidth]{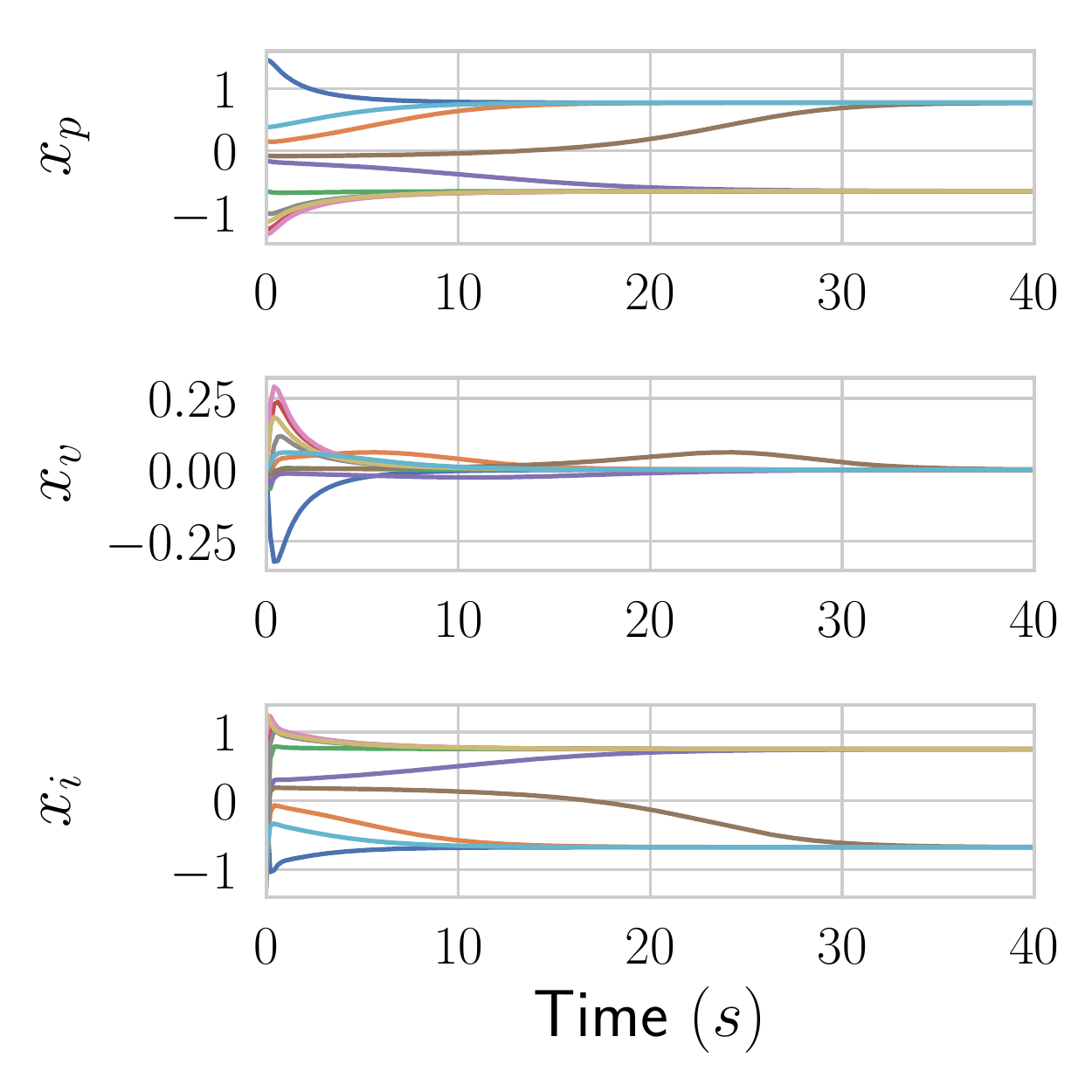}
	\caption{
	\textbf{Left:} Cone of \eqref{eq:spring}, $-2 \leq \alpha' \leq 5$, $0 \leq k_p' \leq 3$.
	\textbf{Right:} Random trajectories of \eqref{eq:spring} with $\alpha(x_p) = 5x_p-7\tanh(x_p)$ and $k_p(x_{ref}-x_p) = 0.1-x_p$. Each color corresponds to a trajectory.
	}
	\label{fig:spring_cone}
	\label{fig:spring_traj}
\end{figure}

Following the discussion on robustness at the end of Section \ref{sec:conical_approx}, the cone found can also be leveraged to explore K-cooperativity for different sets of parameters. 
For example, taking $c=8$, $0 \leq k_p' \leq 3$, and all other parameters as above, one can immediately show (using LP) that system \eqref{eq:spring} is strictly K-cooperative with respect to the same $\cKR(R)$ for the wider range $-2.6 < \alpha' < 6.1$.

%%%%%%%%%%%%%%%%%%%%%%%%%%%%%%%%%%%%%%%%%%%%%%%%%%%%%%%%%%%%%%%%%%%%%%%%%%%%%%%%
\section{CONCLUSIONS}

K-cooperativity combined with a cone finding algorithm
make the analysis of multi-stable nonlinear systems accessible.
The approach was illustrated on two examples, 
showing the potential of the theory in application. 
The performance of Algorithm \ref{alg:sat} strongly depends on the analyzed dynamics. 
For example, the cone found for \eqref{eq:cons_nonlinear} had 52 rays but the one for \eqref{eq:cons_nonlinear_two-state_standard} had 11738 when $\tau = 0.3$ (10 hours of computation), and a cone with 4278 rays was found when $\tau = 0.1$ (20 minutes of computation).
Generally, matrices with a small gap between their dominant eigenvalue and their complex eigenvalues (in the conical relaxation \eqref{eq:coni-cond}) lead to cones with a large number of rays.
Future work will focus on scaling up the approach to systems of large dimension.
Following \cite{forni_path-complete_2017}, we will also study the completeness of the algorithm and we will characterize conditions for its convergence.

%%%%%%%%%%%%%%%%%%%%%%%%%%%%%%%%%%%%%%%%%%%%%%%%%%%%%%%%%%%%%%%%%%%%%%%%%%%%%%%%
%\section{ACKNOWLEDGMENTS}

%%%%%%%%%%%%%%%%%%%%%%%%%%%%%%%%%%%%%%%%%%%%%%%%%%%%%%%%%%%%%%%%%%%%%%%%%%%%%%%%

%\bibliographystyle{ieeetr}
\bibliographystyle{abbrv}

\bibliography{19_CDC-all}

\begin{thebibliography}{10}

\bibitem{angeli_detection_2004}
D.~Angeli, J.~E. Ferrell, and E.~D. Sontag.
\newblock Detection of multistability, bifurcations, and hysteresis in a large
  class of biological positive-feedback systems.
\newblock {\em Proceedings of the National Academy of Sciences of the United
  States of America}, 101(7):1822--1827, 2004.

\bibitem{angeli_monotone_2003}
D.~Angeli and E.~D. Sontag.
\newblock Monotone control systems.
\newblock {\em IEEE Transactions on Automatic Control}, 48(10):1684--1698,
  2003.

\bibitem{angeli_multi-stability_2004}
D.~Angeli and E.~D. Sontag.
\newblock Multi-stability in monotone input/output systems.
\newblock {\em Systems \& Control Letters}, 51(3):185--202, 2004.

\bibitem{berman_nonnegative_1989}
A.~Berman, M.~Neumann, and R.~J. Stern.
\newblock {\em Nonnegative {{Matrices}} in {{Dynamic Systems}}}.
\newblock {Wiley}, 1989.

\bibitem{blanchini_structural_2014}
F.~Blanchini, E.~Franco, and G.~Giordano.
\newblock A {{Structural Classification}} of {{Candidate Oscillatory}} and
  {{Multistationary Biochemical Systems}}.
\newblock {\em Bulletin of Mathematical Biology}, 76(10):2542--2569, 2014.

\bibitem{bushell_hilberts_1973}
P.~J. Bushell.
\newblock Hilbert's metric and positive contraction mappings in a {{Banach}}
  space.
\newblock {\em Archive for Rational Mechanics and Analysis}, 52:330--338, 1973.

\bibitem{farina_positive_2000}
L.~Farina and S.~Rinaldi.
\newblock {\em Positive Linear Systems: Theory and Applications}.
\newblock Pure and Applied Mathematics. {Wiley}, 2000.

\bibitem{forni_differential_2015}
F.~Forni.
\newblock Differential positivity on compact sets.
\newblock In {\em 54th {{IEEE Conference}} on {{Decision}} and {{Control}}
  ({{CDC}})}, pages 6355--6360, 2015.

\bibitem{forni_path-complete_2017}
F.~Forni, R.~M. Jungers, and R.~Sepulchre.
\newblock Path-complete positivity of switching systems.
\newblock {\em 20th IFAC World Congress}, 50(1):4558--4563, 2017.

\bibitem{forni_differential_2014}
F.~Forni and R.~Sepulchre.
\newblock Differential analysis of nonlinear systems: {{Revisiting}} the
  pendulum example.
\newblock In {\em 53rd {{IEEE Conference}} on {{Decision}} and {{Control}}},
  pages 3848--3859, 2014.

\bibitem{forni_differential_2014-1}
F.~Forni and R.~Sepulchre.
\newblock A {{Differential Lyapunov Framework}} for {{Contraction Analysis}}.
\newblock {\em IEEE Transactions on Automatic Control}, 59(3):614--628, 2014.

\bibitem{forni_differentially_2016}
F.~Forni and R.~Sepulchre.
\newblock Differentially {{Positive Systems}}.
\newblock {\em IEEE Transactions on Automatic Control}, 61(2):346--359, 2016.

\bibitem{forni_differential_2018}
F.~Forni and R.~Sepulchre.
\newblock Differential {{Dissipativity Theory}} for {{Dominance Analysis}}.
\newblock {\em Accepted, IEEE Transactions on Automatic Control}, 2018.

\bibitem{fukuda_lecture_2016}
K.~Fukuda.
\newblock Lecture: {{Polyhedral Computation}}.
\newblock 2016.

\bibitem{hirsch_chapter_2006}
M.~W. Hirsch and H.~L. Smith.
\newblock {Chapter 4 Monotone Dynamical Systems}.
\newblock {\em Handbook of Differential Equations: Ordinary Differential
  Equations}, pages 239--357, 2006.

\bibitem{leenheer_tutorial_2004}
P.~D. Leenheer, D.~Angeli, and E.~D. Sontag.
\newblock {\em A {{Tutorial}} on {{Monotone Systems}} - {{With}} an
  {{Application}} to {{Chemical Reaction Networks}}}.
\newblock 2004.

\bibitem{lohmiller_contraction_1998-1}
W.~Lohmiller and J.~Slotine.
\newblock On {{Contraction Analysis}} for {{Non}}-linear {{Systems}}.
\newblock {\em Automatica}, 34(6):683--696, 1998.

\bibitem{luenberger_introduction_1979}
D.~G. Luenberger.
\newblock {\em Introduction to Dynamic Systems: Theory, Models, and
  Applications}.
\newblock {Wiley}, 1979.

\bibitem{mallet-paret_poincare-bendixson_1990}
J.~{Mallet-Paret} and H.~L. Smith.
\newblock The {{Poincare}}-{{Bendixson}} theorem for monotone cyclic feedback
  systems.
\newblock {\em Journal of Dynamics and Differential Equations}, 2(4):367--421,
  1990.

\bibitem{olfati-saber_consensus_2007}
R.~{Olfati-Saber}, J.~A. Fax, and R.~M. Murray.
\newblock Consensus and {{Cooperation}} in {{Networked Multi}}-{{Agent
  Systems}}.
\newblock {\em Proceedings of the IEEE}, 95(1):215--233, 2007.

\bibitem{protasov_when_2010}
V.~Y. Protasov.
\newblock When do several linear operators share an invariant cone?
\newblock {\em Linear Algebra and its Applications}, 433(4):781--789, 2010.

\bibitem{rantzer_scalable_2015}
A.~Rantzer.
\newblock Scalable control of positive systems.
\newblock {\em European Journal of Control}, 24:72--80, 2015.

\bibitem{ren_distributed_2008}
W.~Ren and R.~W. Beard.
\newblock {\em Distributed Consensus in Multi-Vehicle Cooperative Control:
  Theory and Applications}.
\newblock Communications and Control Engineering. {Springer}, 2008.

\bibitem{russo_global_2010}
G.~Russo, M.~di~Bernardo, and E.~D. Sontag.
\newblock Global {{Entrainment}} of {{Transcriptional Systems}} to {{Periodic
  Inputs}}.
\newblock {\em PLOS Computational Biology}, 6(4):e1000739, 2010.

\bibitem{sepulchre_consensus_2010}
R.~Sepulchre, A.~Sarlette, and P.~Rouchon.
\newblock Consensus in non-commutative spaces.
\newblock In {\em 49th {{IEEE Conference}} on {{Decision}} and {{Control}}
  ({{CDC}})}, pages 6596--6601, 2010.

\bibitem{simpson-porco_contraction_2014}
J.~{Simpson-Porco} and F.~Bullo.
\newblock Contraction theory on {{Riemannian}} manifolds.
\newblock {\em Systems \& Control Letters}, 65:74--80, 2014.

\bibitem{smith_monotone_1995}
H.~L. Smith.
\newblock {\em Monotone {{Dynamical Systems}}: {{An Introduction}} to the
  {{Theory}} of {{Competitive}} and {{Cooperative Systems}}}, volume~41 of {\em
  Mathematical {{Surveys}} and {{Monographs}}}.
\newblock {American Mathematical Society}, 1995.

\bibitem{vandergraft_spectral_1968}
J.~S. Vandergraft.
\newblock Spectral {{Properties}} of {{Matrices}} which {{Have Invariant
  Cones}}.
\newblock {\em SIAM Journal on Applied Mathematics}, 16(6):1208--1222, 1968.

\end{thebibliography}

\end{document}